\documentclass[12pt,twoside]{amsart}
\usepackage{amssymb,amsmath,amsthm, amscd, enumerate, mathrsfs}
\usepackage{graphicx, hhline}
\usepackage[all]{xy}
\usepackage[usenames]{color}
\usepackage{hyperref}
\usepackage{fancyhdr}
\hypersetup{colorlinks=true}

\title{On the non-vanishing conjecture and existence of log minimal models}
\author{Kenta Hashizume}
\date{2016/11/20, version 0.03}
\address{Department of Mathematics, Graduate School of Science, 
Kyoto University, Kyoto 606-8502, Japan}
\email{hkenta@math.kyoto-u.ac.jp}

\newtheorem{thm}{Theorem}[section]
\newtheorem{lem}[thm]{Lemma}
\newtheorem{cor}[thm]{Corollary}

\newtheorem{conj}[thm]{Conjecture}

\theoremstyle{definition}
\newtheorem{defin}[thm]{Definition}
\newtheorem{rem}[thm]{Remark}
\newtheorem*{ack}{Acknowledgments} 
\newtheorem{say}[thm]{}
\newtheorem{step}{Step}
\newtheorem{step2}{Step}
\newtheorem{case}{Case}

\begin{document}

\maketitle

\begin{abstract}
We prove that the non-vanishing conjecture and the log minimal model conjecture for projective log canonical pairs can be reduced to the non-vanishing conjecture for smooth projective varieties such that the boundary divisor is zero. 
\end{abstract}

\tableofcontents

\section{Introduction}
Throughout this paper we will work over the complex number field, and we denote Conjecture $\bullet$ with ${\rm dim}\, X = n$ (resp.~${\rm dim}\, X \leq n$) by Conjecture $\bullet_{n}$ (resp.~Conjecture $\bullet_{\leq n}$).  

In this paper we deal with the following two conjectures.

\begin{conj}[Non-vanishing]\label{conjnonvanish}
Let $(X,\Delta)$ be a projective log canonical pair. 
If $K_{X}+\Delta$ is pseudo-effective, then there is an effective $\mathbb{R}$-divisor $D$ such that $K_{X}+\Delta \sim_{\mathbb{R}}D$.
\end{conj}

\begin{conj}[Existence of log minimal model]\label{conjlogminimalmodel}
Let $(X,\Delta)$ be a projective log canonical pair. 
If $K_{X}+\Delta$ is pseudo-effective, then $(X,\Delta)$ has a log minimal model. 
\end{conj}

Birkar \cite{birkar-existII} proved that Conjecture \ref{conjnonvanish}$_{n}$ implies Conjecture \ref{conjlogminimalmodel}$_{n}$. 
On the other hand, Gongyo \cite{gongyo-nonvanishing} proved that Conjecture \ref{conjnonvanish}$_{n}$ for Kawamata log terminal pairs with boundary $\mathbb{Q}$-divisors implies Conjecture \ref{conjnonvanish}$_{n}$ for log canonical pairs with boundary $\mathbb{R}$-divisors assuming the abundance conjecture for $d$-dimensional log canonical pairs with $d\leq n-1$ . 
Today Conjecture \ref{conjnonvanish}$_{\leq3}$ and Conjecture \ref{conjlogminimalmodel}$_{\leq 4}$ is proved (Conjecture \ref{conjlogminimalmodel}$_{4}$ was first proved by Shokurov \cite{shokurov} and another proof was given by Birkar \cite{birkarzariski}) but Conjecture \ref{conjnonvanish} and Conjecture \ref{conjlogminimalmodel} are still open in higher dimension. 

%Siu \cite{siu} announced outline of the proof of the generalized abundance conjecture for compact smooth complex variety, which is stronger than Conjecture \ref{conjnonvanish}, but the details are still not completed. 
In this paper we study the relation between the above two conjectures and the following special case of Conjecture \ref{conjnonvanish}. 

\begin{conj}[Non-vanishing for smooth varieties]\label{conjnonvanishsmooth}
Let $X$ be a smooth projective variety. 
If $K_{X}$ is pseudo-effective, then there is an effective $\mathbb{Q}$-divisor $D$ such that $K_{X} \sim_{\mathbb{Q}}D$.
\end{conj} 

The following theorem is the main result of this paper. 

\begin{thm}\label{thmmain}
Conjecture \ref{conjnonvanishsmooth}$_{n}$ implies Conjecture \ref{conjnonvanish}$_{\leq n}$ and Conjecture \ref{conjlogminimalmodel}$_{\leq  n}$. 
\end{thm}

We remark that in Theorem \ref{thmmain} we do not have any assumptions about the abundance conjecture.  
The proof of Theorem \ref{thmmain} heavily depends on the arguments in \cite{hashizume}. 
%In \cite{hashizume} all arguments were carried out in the framework of log canonical pairs with boundary $\mathbb{Q}$-divisors, but in this paper the boundary divisor of any log canonical pair may be an $\mathbb{R}$-divisor. 
A key ingredient is construction of fibrations with relatively trivial log canonical divisors. 
More precisely, for a given log canonical pair $(X,\Delta)$ and a general ample divisor $A$, we observe behavior of the pseudo-effective threshold $\tau(X,t\Delta;A)$ as a function of $t$. 
When $K_{X}+\Delta$ is pseudo-effective and $t$ moves in $[\epsilon,1]$ for an $\epsilon>0$ sufficiently close to one, we see that $\tau(X,t\Delta;A)$ is a linear function of $t$ (see Remark \ref{rempethre}). 
By using this observation we construct the fibrations mentioned as above. 
For details, see the proof of Theorem \ref{thmmain}. 

From Theorem \ref{thmmain} we immediately obtain the following corollaries. 

\begin{cor}\label{cor1}
Conjecture \ref{conjnonvanish}$_{n}$ and Conjecture \ref{conjnonvanishsmooth}$_{n}$ are equivalent. 
\end{cor}

\begin{cor}\label{cor2}
Conjecture \ref{conjnonvanishsmooth}$_{n}$ implies Conjecture \ref{conjlogminimalmodel}$_{n}$. 
\end{cor}

Corollary \ref{cor1} is a generalization of \cite[Theorem 1.5]{gongyo-nonvanishing} and \cite[Theorem 8.8]{dhp}, and Corollary \ref{cor2} is a generalization of \cite[Theorem 1.4]{birkar-existII}. 
We emphasize that by Corollary \ref{cor2} we can reduce the log minimal model conjecture for log canonical pairs to the non-vanishing conjecture in very simple situation. 

The contents of this paper are as follows: 
In Section \ref{secpre}, we collect some notations and definitions, and we recall two  important theorems (cf.~Theorem \ref{thmacclct} and Theorem \ref{thmglobalacc}). 
In Section \ref{secmain} we prove Theorem \ref{thmmain}, Corollary \ref{cor1} and Corollary \ref{cor2}. 

\begin{ack}
The author was partially supported by JSPS KAKENHI Grant Number JP16J05875 from JSPS. 
The author would like to thank his supervisor Osamu Fujino for useful advice. 
He is grateful to Professor Caucher Birkar for comments. 
He also thank the referee for useful comments and suggestions. 
\end{ack}

\section{Preliminaries}\label{secpre}
In this section we collect some notations and definitions. 
We will freely use the notations and definitions in \cite{bchm}. 
%We can also find in \cite{hashizume} the basic definitions of divisors. 

\begin{say}[Maps]
Let $f:X\to Y$ be a morphism of normal projective varieties. 
Then $f$ is a {\em contraction} if $f$ is surjective and it has connected fibers. 

Let $f:X\dashrightarrow Y$ be a birational map of normal projective varieties. 
Then $f$ is a {\em birational contraction} if $f^{-1}$ does not contract any divisors. 
Let $D$ be an $\mathbb{R}$-divisor on $X$. 
Unless otherwise stated, we mean $f_{*}D$ by denoting $D_{Y}$. 
\end{say}

\begin{say}[Singularities of pairs]
A {\em pair} $(X,\Delta)$ consists of a normal variety $X$ and a boundary $\mathbb{R}$-divisor $\Delta$, that is, an $\mathbb{R}$-divisor whose coefficients belong to $[0,1]$, on $X$ such that $K_{X}+\Delta$ is $\mathbb{R}$-Cartier. 

Let $(X, \Delta)$ be a pair and $f:Y\to X$ be a log resolution of $(X, \Delta)$. 
Then we can write 
$$K_Y=f^*(K_X+\Delta)+\sum_{i} a(E_{i}, X,\Delta) E_i$$ 
where $E_{i}$ are prime divisors on $Y$ and $a(E_{i}, X,\Delta)$ is a real number for any $i$. 
Then we call $a(E_{i}, X,\Delta)$ the {\em discrepancy} of $E_{i}$ with respect to $(X,\Delta)$. 
The pair $(X, \Delta)$ is called {\it Kawamata log terminal} ({\it klt}, for short) if $a(E_{i}, X, \Delta) > -1$ for any log resolution $f$ of $(X, \Delta)$ and any $E_{i}$ on $Y$. 
$(X, \Delta)$ is called {\it log canonical} ({\it lc}, for short) if $a(E_{i}, X, \Delta) \geq -1$ for any log resolution $f$ of $(X, \Delta)$ and any $E_{i}$ on $Y$. 
$(X, \Delta)$ is called {\it divisorial log terminal} ({\it dlt}, for short) if  there exists a log resolution $f:Y \to X$ of $(X, \Delta)$ such that $a(E, X, \Delta) > -1$ for any $f$-exceptional prime divisor $E$ on $Y$.
\end{say}

Next we introduce the definition of some models. 
For some remarks of the models, see \cite[Remark 2.7]{hashizume}.

\begin{defin}[Weak lc models and log minimal models]\label{defnmodels}
Let $(X,\Delta)$ be a projective log canonical pair and $\phi:X \dashrightarrow X'$ be a birational map to a normal projective variety $X'$. 
Let $E$ be the reduced $\phi^{-1}$-exceptional divisor on $X'$, that is, $E=\sum E_{j}$ where $E_{j}$ are $\phi^{-1}$-exceptional prime divisors on $X'$. 
Then the pair $(X', \Delta'=\phi_{*}\Delta+E)$ is called a {\em log birational model} of $(X,\Delta)$. 
A log birational model $(X', \Delta')$ of $(X,\Delta)$ is a {\em weak log canonical model} ({\em weak lc model}, for short) if 
\begin{itemize}
\item
$K_{X'}+\Delta'$ is nef, and 
\item
for any prime divisor $D$ on $X$ which is exceptional over $X'$, we have
$$a(D, X, \Delta) \leq a(D, X', \Delta').$$ 
\end{itemize}
A weak lc model $(X',\Delta')$ of $(X,\Delta)$ is a {\em log minimal model} if 
\begin{itemize}
\item
$X'$ is $\mathbb{Q}$-factorial, and 
\item
the above inequality on discrepancies is strict. 
\end{itemize}
A log minimal model $(X',\Delta')$ of $(X, \Delta)$ is called a {\em good minimal model} if $K_{X'}+\Delta'$ is semi-ample.    
\end{defin}

\begin{defin}[Mori fiber spaces]\label{defnmorifiberspace} 
Let $(X,\Delta)$ be a projective log canonical pair and $(X',\Delta')$ be a log birational model of $(X,\Delta)$. 

Then $(X',\Delta')$ is called a {\em Mori fiber space} if $X'$ is $\mathbb{Q}$-factorial and there is a contraction $X' \to W$ with ${\rm dim}\,W<{\rm dim}\,X'$ such that 
\begin{itemize}
\item
the relative Picard number $\rho(X'/W)$ is one and $-(K_{X'}+\Delta')$ is ample over $W$, and 
\item
for any prime divisor $D$ over $X$, we have
$$a(D,X,\Delta)\leq a(D,X',\Delta')$$
and strict inequality holds if $D$ is a divisor on $X$ and exceptional over $X'$.
\end{itemize} 
\end{defin}

Note that our definition of log minimal models is slightly different from that of \cite%[Definition 2.1]
{birkar-flip}. 
%In \cite{birkar-flip}, log minimal models are supposed to be dlt. 
%i.e.,
The difference is we do not assume that log minimal models are dlt. 
But this difference is intrinsically not important (cf.~\cite[Remark 2.7]{hashizume}). 
In our definition, any weak lc model $(X',\Delta')$ of a $\mathbb{Q}$-factorial lc pair $(X,\Delta)$ constructed with the $(K_{X}+\Delta)$-MMP
%obtained by running the $(K_{X}+\Delta)$-MMP on $\mathbb{Q}$-factorial lc pairs are
is a log minimal model of $(X,\Delta)$ even though $(X',\Delta')$ may not be dlt.

%\begin{defn}[Log smooth models, cf.~{\cite[Definition 2.3]{birkar-flip},\,\cite[Definition 2.12]{hashizume}}]\label{defnlogsmoothmodel}Let $(X,\Delta)$ be a log canonical pair and $f:Y \to X$ be a log resolution of $(X,\Delta)$. Let $\Gamma$ be a boundary $\mathbb{R}$-divisor on $Y$ such that $(Y,\Gamma)$ is log smooth. Then $(Y,\Gamma)$ is a {\em log smooth model} of $(X,\Delta)$ if we write $$K_{Y}+\Gamma=f^{*}(K_{X}+\Delta)+F, $$then\begin{enumerate}\item[(i)]$F$ is an effective $f$-exceptional divisor, and \item[(ii)] every $f$-exceptional prime divisor $E$ satisfying $a(E,X,\Delta)>-1$ is a component of $F$ and $\Gamma-\llcorner \Gamma \lrcorner$, where $\llcorner \Gamma \lrcorner$ is the reduced part of $\Gamma$.  \end{enumerate}\end{defn}

Finally we introduce the definition of log canonical thresholds and pseudo-effective thresholds, and two important theorems which are proved by Hacon, M\textsuperscript{c}Kernan and Xu \cite{hmx-acc}. 
%Those theorems play important roles in the proof of Theorem \ref{thmmain}. 

\begin{defin}[Log canonical thresholds, {cf.~\cite{hmx-acc}}]\label{defnlct}
Let $(X,\Delta)$ be a log canonical pair and let $M\neq0$ be an effective $\mathbb{R}$-Cartier $\mathbb{R}$-divisor. 
Then the {\em log canonical threshold} of $M$ with respect to $(X,\Delta)$, denoted by ${\rm lct}(X,\Delta;M)$, is 
$${\rm lct}(X,\Delta;M)={\rm sup}\{t\in \mathbb{R}\mid (X,\Delta+tM) {\rm \; is \; log \;canonical}\}.$$ 
\end{defin}

\begin{defin}[Pseudo-effective thresholds]
Let $(X,\Delta)$ be a projective log canonical pair and $M$ be an effective $\mathbb{R}$-Cartier $\mathbb{R}$-divisor such that $K_{X}+\Delta+tM$ is pseudo-effective for some $t \geq 0$. 
Then the {\em pseudo-effective threshold} of $M$ with respect to $(X,\Delta)$, denoted by $\tau(X,\Delta;M)$, is 
$$\tau(X,\Delta;M)={\rm inf}\{t\in \mathbb{R}_{\geq 0}\mid K_{X}+\Delta+tM {\rm \; is \; pseudo\mathchar`-effective}\}.$$ 
\end{defin}

\begin{thm}[ACC for log canonical thresholds, {cf. \cite[Theorem 1.1]{hmx-acc}}]\label{thmacclct}
Fix a positive integer $n$, a set $I \subset [0,1]$ and a set $J\subset \mathbb{R}_{>0}$, where $I$ and $J$ satisfy the DCC. 
Let $\mathfrak{T}_{n}(I)$ be the set of log canonical pairs $(X,\Delta)$, where $X$ is a variety of dimension $n$ and the coefficients of $\Delta$ belong to $I$. 
Then the set 
$$\{ {\rm lct}(X,\Delta;M) \mid (X,\Delta) \in \mathfrak{T}_{n}(I), {\rm \; the \;coefficients \; of\; }M{\rm \;belong\; to\;}J \}$$
satisfies the ACC.
\end{thm}

\begin{thm}[ACC for numerically trivial pairs, {cf. \cite[Theorem D]{hmx-acc}}]\label{thmglobalacc}
Fix a positive integer $n$ and a set $I \subset [0,1]$, which satisfies the DCC. 

Then there is a finite set $I_{0}\subset I$ with the following property:

If $(X,\Delta)$ is a log canonical pair such that 
\begin{enumerate}
\item[(i)]
X is projective of dimension $n$, 
\item[(ii)]
the coefficients of $\Delta$ belong to $I$, and
\item[(iii)]
$K_{X}+\Delta$ is numerically trivial,
\end{enumerate}
then the coefficients of $\Delta$ belong to $I_{0}$. 
\end{thm}

\section{Proof of Theorem \ref{thmmain} and corollaries}\label{secmain}
In this section we prove Theorem \ref{thmmain}, Corollary \ref{cor1} and Corollary \ref{cor2}. 

First we recall the following theorem proved by Birkar, which plays a crucial role in the proof of Theorem \ref{thmmain}. 

\begin{thm}[cf.~{\cite[Corollary 1.7]{birkar-existII}}]\label{thmbirkar}
Fix a positive integer $d$, and assume Conjecture \ref{conjlogminimalmodel}$_{\leq d-1}$. 
Let $(X,\Delta)$ be a $d$-dimensional projective log canonical pair such that $K_{X}+\Delta \sim_{\mathbb{R}}D$ for an effective $\mathbb{R}$-divisor $D$. 
Then Conjecture \ref{conjlogminimalmodel} holds for $(X,\Delta)$. 
\end{thm} 

The following lemma is known to the experts, but we write details of proof for reader's convenience. 

\begin{lem}\label{lem1}
Conjecture \ref{conjnonvanishsmooth}$_{n}$ implies Conjecture \ref{conjnonvanishsmooth}$_{\leq n}$. 
\end{lem}

\begin{proof}
Assume Conjecture \ref{conjnonvanishsmooth}$_{n}$ and pick any $d\leq n$.
Let $X$ be a smooth projective variety of dimension $d$ such that $K_{X}$ is pseudo-effective. 
Let $W$ be the product of $X$ and an $(n-d)$-dimensional abelian variety, 
and let $f:W \to X$ be the projection. 
Then $K_{W}=f^{*}K_{X}$ and  $K_{W}$ is pseudo-effective. 
Since we assume Conjecture \ref{conjnonvanishsmooth}$_{n}$, 
Conjecture \ref{conjnonvanishsmooth} holds for $W$, and therefore 
Conjecture \ref{conjnonvanishsmooth} holds for $X$. 
So we are done. 
\end{proof}

From now on we prove Theorem \ref{thmmain}. 
We fix $n$ in Theorem \ref{thmmain}. 

\begin{proof}[Proof of Theorem \ref{thmmain}]
By Lemma \ref{lem1} we may assume Conjecture \ref{conjnonvanishsmooth}$_{\leq n}$. 
Pick any positive integer $d\leq n$. 
We prove that Conjecture \ref{conjnonvanish}$_{d}$ and Conjecture \ref{conjlogminimalmodel}$_{d}$ hold under the assumption that Conjecture \ref{conjnonvanish}$_{\leq d-1}$ and Conjecture \ref{conjlogminimalmodel}$_{\leq d-1}$ hold. 
If we can prove this then %we see that Conjecture \ref{conjlogminimalmodel}$_{d}$ follows from Conjecture \ref{conjnonvanish}$_{\leq d-1}$ and Conjecture \ref{conjlogminimalmodel}$_{\leq d-1}$ (cf.~Theorem \ref{thmbirkar}), and therefore 
Theorem \ref{thmmain} immediately follows. 
By Theorem \ref{thmbirkar} we see that Conjecture \ref{conjlogminimalmodel}$_{\leq d-1}$ and Conjecture \ref{conjnonvanish}$_{d}$ imply Conjecture \ref{conjlogminimalmodel}$_{d}$.  
Therefore it is sufficient to prove that Conjecture \ref{conjnonvanish}$_{\leq d-1}$ and Conjecture \ref{conjlogminimalmodel}$_{\leq d-1}$ imply Conjecture \ref{conjnonvanish}$_{d}$. 

%We can also assume Conjecture \ref{conjlogminimalmodel}$_{\leq d-1}$ by Theorem \ref{thmbirkar}. 

Let $(X,\Delta)$ be a $d$-dimensional lc pair. 
%By Theorem \ref{thmbirkar}, we only have to prove that Conjecture \ref{conjnonvanish} holds for $(X,\Delta)$. 
By taking a dlt blow-up, we can assume that $(X,\Delta)$ is a  $\mathbb{Q}$-factorial dlt pair. 
We can write $\Delta=S+B$, where $S$ is the reduced part of $\Delta$ and $B=\Delta-S$. 
Then we have following two cases. 

\begin{case}\label{case1}
$S\neq0$ and $\tau(X,B\,;S)=1$, where $\tau(X,B\,;S)$ is the pseudo-effective threshold of $S$ with respect to $(X,B)$.
\end{case}
\begin{case}\label{case2}
$S \neq0$ and $\tau(X,B\,;S)<1$, or $S=0$.  
\end{case}

\begin{proof}[Proof of Case \ref{case1}]
We prove it with several steps.
\begin{step2}\label{step0lc}
From this step to Step \ref{step3lc}, we prove that Conjecture \ref{conjlogminimalmodel} holds for $(X,\Delta)$. 

We run the $(K_{X}+\Delta)$-MMP with scaling of an ample divisor $H$ 
$$(X,\Delta)\dashrightarrow \cdots \dashrightarrow (X^{i},\Delta_{X^{i}})\dashrightarrow \cdots.$$ 
Then for any $i$, the birational map $X\dashrightarrow \cdots \dashrightarrow X^{i}$ is also a finitely many steps of the $(K_{X}+\Delta-tS)$-MMP for any sufficiently small $t>0$.  
Since $K_{X}+\Delta-tS$ is not pseudo-effective by hypothesis, we see that $S_{X^{i}}\neq0$ and $\tau(X^{i},B_{X^{i}}\,;S_{X^{i}})=1$ for any $i$. 
Therefore we can replace $(X,\Delta)$ with $(X^{i},\Delta_{X^{i}})$ for some $i\gg0$  and we may assume that there is a big divisor $H$ such that $K_{X}+\Delta+\delta H$ is movable for any sufficiently small $\delta>0$. 
\end{step2}
 
\begin{step2}\label{step1lc} 
Fix $A\geq0$ a general ample $\mathbb{Q}$-divisor such that $(X,\Delta+A)$ is lc, $(X,B+A)$ is klt  and $(1/2)A+K_{X}+B$ and $(1/2)A+S$ are both nef. 
Then 
$$K_{X}+tS+B+A=\Bigl(\frac{1}{2}A+K_{X}+B\Bigr)+t\Bigl(\frac{1}{2}A+S \Bigr)+\frac{1}{2}(1-t)A$$
is nef for any $0\leq t \leq 1$. 
Let $\tau_{t}=\tau(X,tS+B\,;A)$ be the pseudo-effective threshold of $A$ with respect to $(X,tS+B)$ for any $0\leq t < 1$. 
By construction we have $0<\tau_{t}\leq1$ for any $t$. 
In this step we prove that there is $0<\epsilon< 1$ such that the divisor
\begin{equation*}
\begin{split}
&K_{X}+(1-t(1-\epsilon))S+B+t\tau_{\epsilon}A\\
&=(1-t)(K_{X}+\Delta)+t(K_{X}+\epsilon S +B +\tau_{\epsilon}A)
\end{split}
\end{equation*}
is not big for any $0\leq t \leq 1$. 

The idea is similar to \cite[Step 2 and Step 3 in the proof of Proposition 5.3]{hashizume} (see also the proof of \cite[Proposition 8.7]{dhp} or \cite[Lemma 3.1]{gongyo-nonvanishing}). 
Let $\{ u_{k}\}_{k\geq1}$ be a strictly increasing infinite sequence of positive real numbers such that $u_{k}<1$ for any $k$ and ${\rm lim}_{k\to \infty}u_{k}=1$. 
For each $k$, we run the $(K_{X}+u_{k}S+B)$-MMP with scaling of $A$. 
Then we get a Mori fiber space $(X,u_{k}S+B) \dashrightarrow (X_{k},u_{k}S_{X_{k}}+B_{X_{k}})\to Z_{k}$. 
By the basic property of the log MMP with scaling, $K_{X_{k}}+u_{k}S_{X_{k}}+B_{X_{k}}+\tau_{u_{k}}A_{X_{k}}$ is trivial over $Z_{k}$. 
By replacing $\{u_{k}\}_{k\geq1}$ with its subsequence we may assume that ${\rm dim}\,Z_{k}$ is constant. 

We note that $S_{X_{k}}\neq 0$ and $S_{X_{k}}$ is ample over $Z_{k}$ by construction. 
Since ${\rm lim}_{k\to \infty}u_{k}=1$, by the ACC for log canonical thresholds (cf.~Theorem \ref{thmacclct}), there are infinitely many indices $k$ such that ${\rm lct}(X_{k},B_{X_{k}};S_{X_{k}})=1$. 
Therefore, by replacing $\{u_{k}\}_{k\geq1}$ with its subsequence, we may assume that $(X_{k},\Delta_{X_{k}})$ is lc for any $k$. 
%We consider the set $$\{\mu\in \mathbb{R}_{\geq0}\mid \mu={\rm lct}(X_{k},B_{X_{k}};S_{X_{k}})\; {\rm for \;some}\; k\geq1\}.$$By construction we have ${\rm lct}(X_{k},B_{X_{k}};S_{X_{k}})\geq t_{k}$ for any $k$. Moreover the coefficients of $B_{X_{k}}$ and $S_{X_{k}}$ belong to a finite set which does not depend on $k$. 
Moreover, by applying the ACC for numerically trivial pairs (cf.~Theorem \ref{thmglobalacc}) to the general fibers of Mori fiber spaces $(X_{k},u_{k}S_{X_{k}}+B_{X_{k}})\to Z_{k}$, we can find an index $k$ such that $K_{X_{k}}+\Delta_{X_{k}}$ is numerically trivial over $Z_{k}$. 
In particular $K_{X_{k}}+\Delta_{X_{k}}$ is trivial over $Z_{k}$.

Set $\epsilon=u_{k}$ for this $k$. 
Then 
$$(1-t)(K_{X_{k}}+\Delta_{X_{k}})+t(K_{X_{k}}+\epsilon S_{X_{k}}+B_{X_{k}}+\tau_{\epsilon}A_{X_{k}})$$ 
is trivial over $Z_{k}$ for any $0\leq t \leq 1$. 
Since ${\rm dim}\,Z_{k}<{\rm dim}\,X_{k}$, we see that $K_{X}+(1-t(1-\epsilon))S+B+t\tau_{\epsilon}A$
is not big. 
Note that $0<\tau_{\epsilon}\leq 1$. 
\end{step2} 

\begin{step2}\label{step2lc}
Set $G=(1-\epsilon)S-\tau_{\epsilon}A$. 
Then $(X,\Delta-tG)$ is klt and $\Delta-tG$ is big for any $0<t\leq1$ because $\Delta-tG=(1-t(1-\epsilon))S+B+t\tau_{\epsilon}A$. 
In this step we show that there is an infinite sequence $\{a_{k}\}_{k\geq1}$ of positive real numbers such that 
\begin{enumerate}
\item[(i)]
$a_{k}<1$ for any $k$ and ${\rm lim}_{k\to \infty}a_{k}=0$, and 
\item[(ii)]
there is a finitely many steps of the $(K_{X}+\Delta-a_{k}G)$-MMP to a good minimal model 
$$(X, \Delta-a_{k}G)\dashrightarrow (X'_{k}, \Delta_{X'_{k}}-a_{k}G_{X'_{k}})$$ 
such that $(X'_{k},\Delta_{X'_{k}})$ is lc and there is a contraction $X'_{k}\to Y_{k}$ to a normal projective variety $Y_{k}$ such that 

\begin{enumerate}
\item[(ii-a)]
${\rm dim}\,Y_{k}<{\rm dim}\,X'_{k}$, and 
\item[(ii-b)]
$K_{X'_{k}}+\Delta_{X'_{k}}-a_{k}G_{X'_{k}}$ is $\mathbb{R}$-linearly equivalent to the pullback of an ample $\mathbb{R}$-divisor on $Y_{k}$ and $K_{X'_{k}}+\Delta_{X'_{k}}\sim_{\mathbb{R},\,Y_{k}}0$. 
\end{enumerate} 
\end{enumerate} 

Since $\tau_{\epsilon}>0$ and $A$ is ample,
%The idea is similar to \cite[Step 1 and Step 5 in the proof of Proposition 5.1]{hashizume}. 
by applying \cite[Corollary 1.1.5]{bchm} (see also \cite[Theorem 8.9]{dhp}), there are countably many birational maps $X \dashrightarrow X'_{k}$ and contractions $X'_{k}\to Y_{k}\, (k=1,2,\cdots)$ to normal projective varieties $Y_{k}$ such that 
for any $0<\alpha\leq1$, there is an index $k$ satisfying the following:
\begin{enumerate}
\item[$(\clubsuit)$]
$(X, \Delta-\alpha G)\dashrightarrow (X'_{k}, \Delta_{X'_{k}}-\alpha G_{X'_{k}})$ is a finitely many steps of the $(K_{X}+\Delta-\alpha G)$-MMP to a good minimal model and $K_{X'_{k}}+\Delta_{X'_{k}}-\alpha G_{X'_{k}}$ is $\mathbb{R}$-linearly equivalent to the pullback of an ample divisor on $Y_{k}$.  
\end{enumerate}
Renumbering $X'_{k}$ and $Y_{k}$ if necessary, we can find an infinite sequence $\{a_{k}\}_{k\geq1}$ of positive real numbers such that 
\begin{enumerate}
\item[(1)]
$a_{k}<1$ for any $k$ and ${\rm lim}_{k\to \infty}a_{k}=0$, and 
\item[(2)]
for any $k$ and any $\alpha>0$ sufficiently close to $a_{k}$, the birational map $(X, \Delta-\alpha G)\dashrightarrow (X'_{k}, \Delta_{X'_{k}}-\alpha G_{X'_{k}})$ 
and the contraction $X'_{k}\to Y_{k}$ satisfy $(\clubsuit)$.  
\end{enumerate} 
By construction ${\rm lct}(X'_{k},B_{X'_{k}}\,;S_{X'_{k}})\geq 1-a_{k}(1-\epsilon)$. 
Since ${\rm lim}_{k\to \infty}a_{k}=0$, by the ACC for log canonical thresholds (cf.~Theorem \ref{thmacclct}), there are infinitely many indices $k$ such that ${\rm lct}(X'_{k},B_{X'_{k}};S_{X'_{k}})=1$. 
Therefore, by replacing $\{a_{k}\}_{k\geq1}$ with its subsequence, we may assume that $(X'_{k},\Delta_{X'_{k}})$ is lc for any $k$. 

We show that $\{a_{k}\}_{k \geq 1}$, $(X, \Delta-a_{k} G)\dashrightarrow (X'_{k}, \Delta_{X'_{k}}-a_{k} G_{X'_{k}})$ and $X'_{k}\to Y_{k}$ satisfy all the conditions stated at the start of this step. 
By construction it is sufficient to check that the contraction $X'_{k}\to Y_{k}$ satisfies  conditions (ii-a) and (ii-b) for any $k$. 
First we recall that $K_{X}+\Delta-a_{k}G$ is not big by Step \ref{step1lc}. 
Therefore ${\rm dim}\,Y_{k}<{\rm dim}\,X'_{k}$ and thus the contraction $X'_{k}\to Y_{k}$ satisfies condition (ii-a) for any $k$. 
Next, by condition (2), we can find a positive real number $\widetilde{\alpha}\neq a_{k}$ sufficiently close to $a_{k}$ such that
 $K_{X'_{k}}+\Delta_{X'_{k}}-\widetilde{\alpha}G_{X'_{k}}\sim_{\mathbb{R},\,Y_{k}}0$. 
Because $K_{X'_{k}}+\Delta_{X'_{k}}$ is represented by a linear combination of $K_{X'_{k}}+\Delta_{X'_{k}}-a_{k} G_{X'_{k}}$ and $K_{X'_{k}}+\Delta_{X'_{k}}-\widetilde{\alpha}G_{X'_{k}}$, we have $K_{X'_{k}}+\Delta_{X'_{k}}\sim_{\mathbb{R},\,Y_{k}}0$. 
Therefore $X'_{k}\to Y_{k}$ satisfies condition (ii-b) for any $k$. 

In this way we see that $\{a_{k}\}_{k \geq 1}$, $(X, \Delta-a_{k} G)\dashrightarrow (X'_{k}, \Delta_{X'_{k}}-a_{k} G_{X'_{k}})$ and $X'_{k}\to Y_{k}$ satisfy all the conditions stated at the start of this step. 
Thus we complete this step.  
\end{step2}

\begin{step2}\label{step2.5lc}
In this step we prove that Conjecture \ref{conjlogminimalmodel} holds for $(X'_{k},\Delta_{X'_{k}})$ for any $k$, where $(X'_{k},\Delta_{X'_{k}})$ was constructed in Step \ref{step2lc}. 
By Theorem \ref{thmbirkar}, we may show that Conjecture \ref{conjnonvanish} holds for $(X'_{k},\Delta_{X'_{k}})$ for any $k$. 
In this step we fix $k$. 

Since $K_{X'_{k}}+\Delta_{X'_{k}}\sim_{\mathbb{R},\,Y_{k}}0$ and $K_{X'_{k}}+\Delta_{X'_{k}}-a_{k}G_{X'_{k}}\sim_{\mathbb{R},\,Y_{k}}0$, we have $G_{X'_{k}}=(1-\epsilon)S_{X'_{k}}-\tau_{\epsilon}A_{X'_{k}}\sim_{\mathbb{R},\,Y_{k}}0$. 
Since $A_{X'_{k}}$ is big, we see that $S_{X'_{k}}$ is big over $Y_{k}$. 
Therefore $S_{X'_{k}}\neq 0$ and some component of $S_{X'_{k}}$ dominates $Y_{k}$ because ${\rm dim}\,Y_{k}<{\rm dim}\,X'_{k}$ by condition (ii-a) in Step \ref{step2lc}. 
%To prove that that Conjecture \ref{conjlogminimalmodel} holds for $(X'_{k},\Delta_{X'_{k}})$, we may show that Conjecture \ref{conjnonvanish} holds for $(X'_{k},\Delta_{X'_{k}})$ by Theorem \ref{thmbirkar}. 
Let $f:(V,\Gamma) \to(X'_{k},\Delta_{X'_{k}})$ be a dlt blow-up and let $T$ be a component of $f_{*}^{-1}S_{X'_{k}}$ dominating $Y_{k}$. 
Then we have $K_{V}+\Gamma \sim_{\mathbb{R},\,Y_{k}}0$. 
Let $M$ be an $\mathbb{R}$-divisor on $Y_{k}$ such that $K_{V}+\Gamma$ is $\mathbb{R}$-linearly equivalent to the pullback of $M$. 
Then $M$ is pseudo-effective.
By the adjunction $(T,{\rm Diff}(\Gamma-T))$ is dlt, and $K_{T}+{\rm Diff}(\Gamma-T)$ is pseudo-effective.
Since we assume Conjecture \ref{conjnonvanish}$_{\leq d-1}$, Conjecture \ref{conjnonvanish} holds for $(T,{\rm Diff}(\Gamma-T))$. 
Then there is an effective $\mathbb{R}$-divisor $E$ on $Y_{k}$ such that $M\sim_{\mathbb{R}}E$. 
Therefore Conjecture \ref{conjnonvanish} holds for $(V,\Gamma)$, and hence Conjecture \ref{conjnonvanish} holds for $(X'_{k},\Delta_{X'_{k}})$. 

In this way we see that  Conjecture \ref{conjlogminimalmodel} holds for $(X'_{k},\Delta_{X'_{k}})$. 
\end{step2}

\begin{step2}\label{step3lc}
In this step we show that Conjecture \ref{conjlogminimalmodel} holds for $(X,\Delta)$. 
%We will freely use the notations of Step \ref{step2lc} in the proof. 

We keep the track of \cite[Step 4 in the proof of Theorem 5.1]{hashizume}. 
By replacing $\{a_{k}\}_{k\geq1}$ with its subsequence, we can assume that $X'_{i}$ and $X'_{j}$ are isomorphic in codimension one for any $i$ and $j$. 
Indeed, any prime divisor $P$ contracted by the birational map $X\dashrightarrow X'_{k}$ is a component of $N_{\sigma}(K_{X}+\Delta-a_{k}G)$. 
But since we have 
$$N_{\sigma}(K_{X}+\Delta-a_{k}G)\leq (1-a_{k})N_{\sigma}(K_{X}+\Delta)+a_{k}N_{\sigma}(K_{X}+\Delta-G),$$
$P$ is also a component of $N_{\sigma}(K_{X}+\Delta)+N_{\sigma}(K_{X}+\Delta-G)$, which does not depend on $k$. 
Thus we can assume that $X'_{i}$ and $X'_{j}$ are isomorphic in codimension one by replacing $\{a_{k}\}_{k\geq1}$ with its subsequence. 

By Step \ref{step2.5lc}, $(X'_{1},\Delta_{X'_{1}})$ has a log minimal model. 
Therefore, by \cite[Theorem 4.1 (iii)]{birkar-flip}, we can run the $(K_{X'_{1}}+\Delta_{X'_{1}})$-MMP with scaling of an ample divisor and  get a log minimal model $(X'_{1},\Delta_{X'_{1}})\dashrightarrow (X'',\Delta_{X''})$. 
Then we can check that $(X'',\Delta_{X''}-tG_{X''})$ is klt and $\Delta_{X''}-tG_{X''}$ is big for any sufficiently small $t>0$. 
Fix a sufficiently small positive real number $t\ll a_{1}$. 
By \cite[Corollary 1.4.2]{bchm} and running the $(K_{X''}+\Delta_{X''}-tG_{X''})$-MMP with scaling of an ample divisor, we can get a log minimal model $(X'',\Delta_{X''}-tG_{X''})\dashrightarrow (X''',\Delta_{X'''}-tG_{X'''})$. 
Since $t>0$ is sufficiently small, by the standard argument of the length of extremal rays (cf.~\cite[Proposition 3.2]{birkar-existII}), $K_{X'''}+\Delta_{X'''}$ is nef.
Now we get the following sequence of birational maps 
$$X\dashrightarrow X'_{1}\dashrightarrow X''\dashrightarrow X'''$$ 
where $X\dashrightarrow X'_{1}$ (resp.~$X'_{1}\dashrightarrow X''$, $X''\dashrightarrow X'''$) is a finitely many steps of the $(K_{X}+\Delta-a_{1}G)$-MMP (resp.~the $(K_{X'_{1}}+\Delta_{X'_{1}})$-MMP, the $(K_{X''}+\Delta_{X''}-tG_{X''})$-MMP) to a log minimal model. 

We can show that $X'_{1}$ and $X'''$ are isomorphic in codimension one. 
To see this, we may show that $X'_{1}\dashrightarrow X''$ and $X'' \dashrightarrow X'''$ contain only flips. 
Recall that there is a big divisor $H$ such that $K_{X}+\Delta+\delta H$ in movable for any sufficiently small $\delta>0$, which is stated in Step \ref{step0lc}. 
Since $X \dashrightarrow X'_{1}$ is a birational contraction, $K_{X'_{1}}+\Delta_{X'_{1}}+\delta H_{X'_{1}}$ is movable for any sufficiently small $\delta>0$. 
Then $N_{\sigma}(K_{X'_{1}}+\Delta_{X'_{1}})=0$ and thus $X'_{1}\dashrightarrow X''$  contains only flips. 
Furthermore we see that $N_{\sigma}(K_{X''}+\Delta_{X''}-a_{1}G_{X''})=0$ since $K_{X'_{1}}+\Delta_{X'_{1}}-a_{1}G_{X'_{1}}$ is semi-ample, which is condition (ii) in Step \ref{step2lc}.  
Now we have $N_{\sigma}(K_{X''}+\Delta_{X''})=0$, which follows from that $K_{X''}+\Delta_{X''}$ is nef. 
From these facts we have
\begin{equation*}
\begin{split}
N_{\sigma}&(K_{X''}+\Delta_{X''}-tG_{X''})\\
&\leq \Bigl(1-\frac{t}{a_{1}}\Bigr)N_{\sigma}(K_{X''}+\Delta_{X''})+\frac{t}{a_{1}}N_{\sigma}(K_{X''}+\Delta_{X''}-a_{1}G_{X''})=0
\end{split}
\end{equation*}
and hence $X'' \dashrightarrow X'''$ contains only flips. 
In this way we see that $X'_{1}$ and $X'''$ are isomorphic in codimension one. 

Since ${\rm lim}_{k\to \infty}a_{k}=0$, we have $t\geq a_{k}$ for any $k\gg 0$. 
Then 
$K_{X'''}+\Delta_{X'''}-a_{k}G_{X'''}$ 
is nef for any $k\gg0$ because $K_{X'''}+\Delta_{X'''}$ and $K_{X'''}+\Delta_{X'''}-tG_{X'''}$ are nef. 
Moreover $X'''$ and $X'_{k}$ are isomorphic in codimension one since $X'_{k}$ and $X'_{1}$ are isomorphic in codimension one and $X'_{1}$ and $X'''$ are isomorphic in codimension one. 
We recall that $(X'_{k},\Delta_{X'_{k}}-a_{k}G_{X'_{k}})$ is a log minimal model of $(X,\Delta-a_{k}G)$, which is condition (ii) in Step \ref{step2lc}. 
From these facts, we see that $(X''',\Delta_{X'''}-a_{k}G_{X'''})$ is a log minimal model of $(X,\Delta-a_{k}G)$ for any $k\gg 0$. 
Let $p:W \to X$ and $q:W \to X'''$ be a common resolution of $X \dashrightarrow X'''$. 
Then for any $k\gg 0$, we have 
$$p^{*}(K_{X}+\Delta-a_{k}G)-q^{*}(K_{X'''}+\Delta_{X'''}-a_{k}G_{X'''})\geq 0. $$
By considering the limit $k\to \infty$, we have 
$$p^{*}(K_{X}+\Delta)-q^{*}(K_{X'''}+\Delta_{X'''})\geq 0.$$ 
Since $K_{X'''}+\Delta_{X'''}$ is nef, we see that $(X''',\Delta_{X'''})$ is a weak lc model of $(X,\Delta)$. 
Therefore, by \cite[corollary 3.7]{birkar-flip}, $(X,\Delta)$ has a log minimal model. 
\end{step2} 

\begin{step2}\label{step4lc}
Finally we prove that Conjecture \ref{conjnonvanish} holds for $(X,\Delta)$. 
By running the $(K_{X}+\Delta)$-MMP with scaling of an ample divisor and replacing $(X,\Delta)$ with the resulting log minimal model, we can assume that $K_{X}+\Delta$ is nef. 
Note that after this process $S\neq0$ and the equation $\tau(X,B\,;S)=1$ 
still holds. 
%Therefore $K_{X}+\Delta-tS$ is not pseudo-effective for any $t>0$. 
Pick a sufficiently small positive real number $t$ and run the $(K_{X}+\Delta-tS)$-MMP with scaling of an ample divisor. 
Then we get a Mori fiber space 
$$(X,\Delta-tS)\dashrightarrow (X',\Delta_{X'}-tS_{X'})\to Z.$$ 
Moreover, since $t$ is sufficiently small, $K_{X'}+\Delta_{X'}$ is trivial over $Z$ and Conjecture \ref{conjnonvanish} for $(X,\Delta)$ is equivalent to Conjecture \ref{conjnonvanish} for $(X',\Delta_{X'})$ (see \cite[Proposition 3.2]{birkar-existII}). 
We also see that there is a component of $S_{X'}$ dominating $Z$ because $S_{X'}$ is ample over $Z$. 
%Let $T$ be a component of $S_{X'}$ dominating $Z$. 
By the same arguments as in Step \ref{step2.5lc} we can prove that Conjecture \ref{conjnonvanish} holds for $(X',\Delta_{X'})$ with the adjunction and Conjecture \ref{conjnonvanish}$_{\leq d-1}$. 
Thus Conjecture \ref{conjnonvanish} holds for $(X,\Delta)$ and so we are done.
\end{step2}
\end{proof}

\begin{proof}[Proof of Case \ref{case2}]
In this case we can assume that $(X,\Delta)$ is klt since we only have to prove that Conjecture \ref{conjnonvanish} holds for $(X,\Delta)$.  
Taking a log resolution of $(X,\Delta)$, we can assume that $X$ is smooth. 
We put $\tau=\tau(X,0\,;\Delta)$. 
Then we may assume that $\Delta \neq0$ and $\tau>0$ because otherwise Conjecture \ref{conjnonvanish} for $(X,\Delta)$ is obvious from Conjecture \ref{conjnonvanishsmooth}$_{\leq n}$. 
Moreover we may assume that $\tau=1$ by replacing $(X,\Delta)$ with $(X,\tau\Delta)$. 

We prove Case \ref{case2} with several steps. 
The proof is very similar to the proof of Case \ref{case1} except Step \ref{step2.5klt}. 
In the rest of the proof we will use the fact that $(X,\Delta)$ is $\mathbb{Q}$-factorial klt but we will not use the assumption that $X$ is smooth.  

\begin{step}\label{step0klt}
From this step to Step \ref{step3klt}, we prove that Conjecture \ref{conjlogminimalmodel} holds for $(X,\Delta)$. 

We run the $(K_{X}+\Delta)$-MMP with scaling of an ample divisor $H$ 
$$(X,\Delta)\dashrightarrow \cdots \dashrightarrow (X^{i},\Delta_{X^{i}})\dashrightarrow \cdots.$$ 
By the same argument as in Step \ref{step0lc} in the proof of  Case \ref{case1}, we can replace $(X,\Delta)$ with $(X^{i},\Delta_{X^{i}})$ for some $i\gg0$ and we may assume that there is a big divisor $H$ such that $K_{X}+\Delta+\delta H$ is movable for any sufficiently small $\delta>0$. 
Note that $\Delta\neq0$ and $\tau(X,0\,;\Delta)=1$ still hold after this process. 
\end{step}

\begin{step}\label{step1klt}
Fix $A\geq0$ a general ample $\mathbb{Q}$-divisor such that $(X,\Delta+A)$ is klt and $(1/2)A+K_{X}$ and $(1/2)A+\Delta$ are both nef. 
Then 
$$K_{X}+t\Delta+A=\Bigl(\frac{1}{2}A+K_{X}\Bigr)+t\Bigl(\frac{1}{2}A+\Delta \Bigr)+\frac{1}{2}(1-t)A$$
is nef for any $0\leq t \leq1$. 
We put $\tau_{t}=\tau(X,t\Delta\,;A)$ for any $0\leq t < 1$. 
By construction we have $0<\tau_{t}\leq1$ for any $t$. 
In this step we prove that there is $0<\epsilon< 1$ such that the divisor
$$K_{X}+(1-t(1-\epsilon))\Delta+t\tau_{\epsilon}A=(1-t)(K_{X}+\Delta)+t(K_{X}+\epsilon\Delta+\tau_{\epsilon}A)$$
is not big for any $0\leq t \leq 1$. 

Pick a strictly increasing infinite sequence $\{ u_{k}\}_{k\geq1}$ of positive real numbers such that $u_{k}<1$ for any $k$ and ${\rm lim}_{k\to \infty}u_{k}=1$. 
For each $k$, run the $(K_{X}+u_{k}\Delta)$-MMP with scaling of $A$. 
Then we get a Mori fiber space $(X,u_{k}\Delta) \dashrightarrow (X_{k},u_{k}\Delta_{X_{k}})\to Z_{k}$. 
By the basic property of the log MMP with scaling, $K_{X_{k}}+u_{k}\Delta_{X_{k}}+\tau_{u_{k}}A_{X_{k}}$ is trivial over $Z_{k}$. 
Now we carry out the same arguments as in Step \ref{step1lc} in the proof of Case \ref{case1}, and we can find an index $k$ such that $(X_{k},\Delta_{X_{k}})$ is lc and  $K_{X_{k}}+\Delta_{X_{k}}$ is numerically trivial over $Z_{k}$ by the ACC for log canonical thresholds (cf.~Theorem \ref{thmacclct}) and the ACC for numerically trivial pairs (cf.~Theorem \ref{thmglobalacc}). 
Set $\epsilon=u_{k}$ for this $k$. 
Then we see that
$$(1-t)(K_{X_{k}}+\Delta_{X_{k}})+t(K_{X_{k}}+\epsilon\Delta_{X_{k}}+\tau_{\epsilon}A_{X_{k}})$$ 
is not big, and therefore $K_{X}+(1-t(1-\epsilon))\Delta+t\tau_{\epsilon}A$
is not big. 
Note that $0<\tau_{\epsilon}\leq 1$. 
\end{step}

\begin{step}\label{step2klt}
Set $G=(1-\epsilon)\Delta-\tau_{\epsilon}A$. 
Then $(X,\Delta-tG)$ is klt and $\Delta-tG$ is big for any $0<t\leq1$ because $\Delta-tG=(1-t(1-\epsilon))\Delta+t\tau_{\epsilon}A$. 
In this step we show that there is an infinite sequence $\{a_{k}\}_{k\geq1}$ of positive real numbers such that 
\begin{enumerate}
\item[(i)]
$a_{k}<1$ for any $k$ and ${\rm lim}_{k\to \infty}a_{k}=0$, and 
\item[(ii)]
there is a finitely many steps of the $(K_{X}+\Delta-a_{k}G)$-MMP to a good minimal model 
$$(X, \Delta-a_{k}G)\dashrightarrow (X'_{k}, \Delta_{X'_{k}}-a_{k}G_{X'_{k}})$$ 
such that $(X'_{k},\Delta_{X'_{k}})$ is lc and there is a contraction $X'_{k}\to Y_{k}$ to a normal projective variety $Y_{k}$ such that 

\begin{enumerate}
\item[(ii-a)]
${\rm dim}\,Y_{k}<{\rm dim}\,X'_{k}$, and 
\item[(ii-b)]
$K_{X'_{k}}+\Delta_{X'_{k}}-a_{k}G_{X'_{k}}$ is $\mathbb{R}$-linearly equivalent to the pullback of an ample $\mathbb{R}$-divisor on $Y_{k}$ and $K_{X'_{k}}+\Delta_{X'_{k}}\sim_{\mathbb{R},\,Y_{k}}0$. 
\end{enumerate} 
\end{enumerate} 

The arguments are very similar to Step \ref{step2lc} in the proof of Case \ref{case1}. 
Since $\tau_{\epsilon}>0$ and $A$ is ample,
%The idea is similar to \cite[Step 1 and Step 5 in the proof of Proposition 5.1]{hashizume}. 
by applying \cite[Corollary 1.1.5]{bchm} (see also \cite[Theorem 8.9]{dhp}), there are countably many birational maps $X \dashrightarrow X'_{k}$ and contractions $X'_{k}\to Y_{k}\, (k=1,2,\cdots)$ to normal projective varieties $Y_{k}$ such that 
for any $0<\alpha\leq1$, there is an index $k$ satisfying the following:
\begin{enumerate}
\item[$(\clubsuit)$]
$(X, \Delta-\alpha G)\dashrightarrow (X'_{k}, \Delta_{X'_{k}}-\alpha G_{X'_{k}})$ is a finitely many steps of the $(K_{X}+\Delta-\alpha G)$-MMP to a good minimal model and $K_{X'_{k}}+\Delta_{X'_{k}}-\alpha G_{X'_{k}}$ is $\mathbb{R}$-linearly equivalent to the pullback of an ample divisor on $Y_{k}$.  
\end{enumerate}
Renumbering $X'_{k}$ and $Y_{k}$ if necessary, we can find an infinite sequence $\{a_{k}\}_{k\geq1}$ of positive real numbers such that 
\begin{enumerate}
\item[(1)]
$a_{k}<1$ for any $k$ and ${\rm lim}_{k\to \infty}a_{k}=0$, and 
\item[(2)]
for any $k$ and any $\alpha>0$ sufficiently close to $a_{k}$, the birational map $(X, \Delta-\alpha G)\dashrightarrow (X'_{k}, \Delta_{X'_{k}}-\alpha G_{X'_{k}})$ 
and the contraction $X'_{k}\to Y_{k}$ satisfy $(\clubsuit)$.  
\end{enumerate} 
By construction ${\rm lct}(X'_{k},0\,;\Delta_{X'_{k}})\geq 1-a_{k}(1-\epsilon)$. 
Since ${\rm lim}_{k\to \infty}a_{k}=0$, by the ACC for log canonical thresholds (cf.~Theorem \ref{thmacclct}), there are infinitely many indices $k$ such that ${\rm lct}(X'_{k},0\,;\Delta_{X'_{k}})=1$. 
Therefore, by replacing $\{a_{k}\}_{k\geq1}$ with its subsequence, we may assume that $(X'_{k},\Delta_{X'_{k}})$ is lc for any $k$. 

Furthermore, by the same arguments as in Step \ref{step2lc} in the proof of Case \ref{case1}, we can check that the contraction $X'_{k}\to Y_{k}$
satisfies conditions (ii-a) and (ii-b) for any $k$. 
Note that $K_{X}+\Delta-a_{k}G$ is not big by Step \ref{step1lc}.  
In this way we see that the sequence $\{a_{k}\}_{k \geq 1}$, the birational maps $(X, \Delta-a_{k} G)\dashrightarrow (X'_{k}, \Delta_{X'_{k}}-a_{k} G_{X'_{k}})$ and contractions $X'_{k}\to Y_{k}$ satisfy all the conditions stated at the start of this step. 
Thus we complete this step.  
\end{step}

\begin{step}\label{step2.5klt}
In this step we prove that Conjecture \ref{conjlogminimalmodel} holds for $(X'_{k},\Delta_{X'_{k}})$ for any $k$, where $(X'_{k},\Delta_{X'_{k}})$ was constructed in Step \ref{step2klt}. 
We note that $(X,\Delta)$ is klt but $(X'_{k},\Delta_{X'_{k}})$ may not be klt. 
By Theorem \ref{thmbirkar}, we only have to show that Conjecture \ref{conjnonvanish} holds for $(X'_{k},\Delta_{X'_{k}})$ for any $k$. 
In this step we fix $k$. 

If $(X'_{k},\Delta_{X'_{k}})$ is klt, by applying Ambro's canonical bundle formula (cf.~\cite[Corollary 3.2]{fg-bundle}) to $X'_{k}\to Y_{k}$ and since we assume Conjecture \ref{conjnonvanish}$_{\leq d-1}$, Conjecture \ref{conjnonvanish} holds for $(X'_{k},\Delta_{X'_{k}})$. 
Therefore we may assume that $(X'_{k},\Delta_{X'_{k}})$ is not klt. 

Let $f:(V,\Gamma)\to (X'_{k},\Delta_{X'_{k}})$ be a dlt blow-up of $(X'_{k},\Delta_{X'_{k}})$ and we write $\Gamma=S_{V}+B_{V}$, where $S_{V}$ is the reduced part of $\Gamma$ and $B_{V}=\Gamma-S_{V}$. 
Then $S_{V}\neq0$ by our assumption. 
We may prove that Conjecture \ref{conjnonvanish} holds for $(V,\Gamma)$. 
%If some components of $S_{V}$ dominates $Y_{n}$, then Conjecture \ref{conjnonvanish} for $(V,\Gamma)$ holds by the adjunction and Conjecture \ref{conjnonvanish}$_{\leq d-1}$ because ${\rm dim}\,Y_{k}<{\rm dim}\,X'_{k}$ (see Step \ref{step2.5lc} in the proof of Case \ref{case1}). Therefore we may assume that $V$ is vertical over $Y_{n}$. 
If $\tau(V,B_{V}\,;S_{V})=1$, then Conjecture \ref{conjnonvanish} holds for $(V,\Gamma)$ by Case \ref{case1}. 
Therefore we may assume that $\tau(V,B_{V}\,;S_{V})<1$. 
Note that $K_{V}+\Gamma \sim_{\mathbb{R},\,Y_{k}}0$ by construction. 

Since $K_{X'_{k}}+\Delta_{X'_{k}}\sim_{\mathbb{R},\,Y_{k}}0$ and $K_{X'_{k}}+\Delta_{X'_{k}}-a_{k}G_{X'_{k}}\sim_{\mathbb{R},\,Y_{k}}0$, we have $G_{X'_{k}}=(1-\epsilon)\Delta_{X'_{k}}-\tau_{\epsilon}A_{X'_{k}}\sim_{\mathbb{R},\,Y_{k}}0$. 
Since $A_{X'_{k}}$ is big, we see that $\Delta_{X'_{k}}$ is big over $Y_{k}$. 
Then $\Gamma$ is also big over $Y_{k}$ because $\Gamma$ contains $f_{*}^{-1}\Delta_{X'_{k}}$ and all $f$-exceptional prime divisors. 
We pick sufficiently small positive real number $t<1$ so that $\tau(V,B_{V}\,;S_{V})\leq 1-t$ and $\Gamma -tS_{V}$ is big over $Y_{k}$. 
Then $(V,\Gamma-tS_{V})$ is klt and $K_{V}+\Gamma -tS_{V}$ is pseudo-effective. 
Moreover we see that $K_{V}+\Gamma -tS_{V}$ is not big over $Y_{k}$ because $K_{V}+\Gamma -tS_{V}\sim_{\mathbb{R},\,Y_{k}}-tS_{V}$ and ${\rm dim}\,Y_{k}<{\rm dim}\,V$. 
By construction it is sufficient to prove that Conjecture \ref{conjnonvanish} holds for $(V,\Gamma-tS_{V})$. 

We run the $(K_{V}+\Gamma-tS_{V})$-MMP over $Y_{k}$ with scaling of an ample divisor. 
By \cite[Corollary 1.4.2]{bchm}, we get a good minimal model $(V,\Gamma-tS_{V})\dashrightarrow (V',\Gamma_{V'}-tS_{V'})$ over $Y_{k}$, where $S_{V'}$ is the birational transform of $S_{V}$ on $V'$. 
Then there is a contraction $V' \to \widetilde{Y}$ to a normal projective variety $\widetilde{Y}$ over $Y_{k}$ such that $K_{V'}+\Gamma_{V'}-tS_{V'}\sim_{\mathbb{R},\,\widetilde{Y}}0$. 
We can check that $(V',\Gamma_{V'}-tS_{V'})$ is klt, and furthermore ${\rm dim}\,\widetilde{Y}<{\rm dim}\,V'$ since $K_{V'}+\Gamma_{V'}-tS_{V'}$ is not big over $Y_{k}$. 
Therefore, applying Ambro's canonical bundle formula (cf.~\cite[Corollary 3.2]{fg-bundle}) to $V'\to \widetilde{Y}$ and since we assume Conjecture \ref{conjnonvanish}$_{\leq d-1}$, Conjecture \ref{conjnonvanish} holds for $(V',\Gamma_{V'}-tS_{V'})$. 
Then Conjecture \ref{conjnonvanish} holds for $(V,\Gamma-tS_{V})$, and thus Conjecture \ref{conjnonvanish} holds for $(V,\Gamma_{V})$. 

In this way we see that Conjecture \ref{conjlogminimalmodel} holds for $(X'_{k},\Delta_{X'_{k}})$ for any $k$, and we complete this step. 
\end{step}

\begin{step}\label{step3klt}
In this step we show that Conjecture \ref{conjlogminimalmodel} holds for $(X,\Delta)$. 
The arguments are very similar to Step \ref{step3lc} in the proof of Case \ref{case1}. 
%We will freely use the notations of Step \ref{step2lc} in the proof. 

%We keep the track of \cite[Step 4 in the proof of Theorem 5.1]{hashizume}. 
By replacing $\{a_{k}\}_{k\geq1}$ with its subsequence, we can assume that $X'_{i}$ and $X'_{j}$ are isomorphic in codimension one for any $i$ and $j$. 
Indeed, any prime divisor $P$ contracted by the birational map $X\dashrightarrow X'_{k}$ is a component of $N_{\sigma}(K_{X}+\Delta-a_{k}G)$. 
But since we have 
$$N_{\sigma}(K_{X}+\Delta-a_{k}G)\leq (1-a_{k})N_{\sigma}(K_{X}+\Delta)+a_{k}N_{\sigma}(K_{X}+\Delta-G),$$
$P$ is also a component of $N_{\sigma}(K_{X}+\Delta)+N_{\sigma}(K_{X}+\Delta-G)$, which does not depend on $k$. 
Thus we can assume that $X'_{i}$ and $X'_{j}$ are isomorphic in codimension one by replacing $\{a_{k}\}_{k\geq1}$ with its subsequence. 

Since $(X'_{1},\Delta_{X'_{1}})$ has a log minimal model, by \cite[Theorem 4.1 (iii)]{birkar-flip}, we can run the $(K_{X'_{1}}+\Delta_{X'_{1}})$-MMP with scaling of an ample divisor and  get a log minimal model $(X'_{1},\Delta_{X'_{1}})\dashrightarrow (X'',\Delta_{X''})$. 
Then we can check that $(X'',\Delta_{X''}-tG_{X''})$ is klt and $\Delta_{X''}-tG_{X''}$ is big for any sufficiently small $t>0$. 
Fix a sufficiently small positive real number $t\ll a_{1}$. 
By \cite[Corollary 1.4.2]{bchm} and running the $(K_{X''}+\Delta_{X''}-tG_{X''})$-MMP with scaling of an ample divisor, we can get a log minimal model $(X'',\Delta_{X''}-tG_{X''})\dashrightarrow (X''',\Delta_{X'''}-tG_{X'''})$. 
Since $t>0$ is sufficiently small, by the standard argument of the length of extremal rays (cf.~\cite[Proposition 3.2]{birkar-existII}), we see that $K_{X'''}+\Delta_{X'''}$ is nef.
Now we get the following sequence of birational maps 
$$X\dashrightarrow X'_{1}\dashrightarrow X''\dashrightarrow X'''$$ 
where $X\dashrightarrow X'_{1}$ (resp.~$X'_{1}\dashrightarrow X''$, $X''\dashrightarrow X'''$) is a finitely many steps of the $(K_{X}+\Delta-a_{1}G)$-MMP (resp.~the $(K_{X'_{1}}+\Delta_{X'_{1}})$-MMP, the $(K_{X''}+\Delta_{X''}-tG_{X''})$-MMP) to a log minimal model. 

We can show that $X'_{1}$ and $X'''$ are isomorphic in codimension one. 
To see this, we may show that $X'_{1}\dashrightarrow X''$ and $X'' \dashrightarrow X'''$ contain only flips. 
Recall that there is a big divisor $H$ such that $K_{X}+\Delta+\delta H$ is movable for any sufficiently small $\delta>0$, which is stated in Step \ref{step0klt} in this proof. 
Since $X \dashrightarrow X'_{1}$ is a birational contraction, $K_{X'_{1}}+\Delta_{X'_{1}}+\delta H_{X'_{1}}$ is movable for any sufficiently small  $\delta>0$. 
Then $N_{\sigma}(K_{X'_{1}}+\Delta_{X'_{1}})=0$ and thus $X'_{1}\dashrightarrow X''$  contains only flips. 
Furthermore we see that $N_{\sigma}(K_{X''}+\Delta_{X''}-a_{1}G_{X''})=0$ since $K_{X'_{1}}+\Delta_{X'_{1}}-a_{1}G_{X'_{1}}$ is semi-ample, which is condition (ii) in Step \ref{step2klt}. 
Now we have $N_{\sigma}(K_{X''}+\Delta_{X''})=0$, which follows from that $K_{X''}+\Delta_{X''}$ is nef. 
From these facts, we have 
\begin{equation*}
\begin{split}
N_{\sigma}&(K_{X''}+\Delta_{X''}-tG_{X''})\\
&\leq \Bigl(1-\frac{t}{a_{1}}\Bigr)N_{\sigma}(K_{X''}+\Delta_{X''})+\frac{t}{a_{1}}N_{\sigma}(K_{X''}+\Delta_{X''}-a_{1}G_{X''})=0
\end{split}
\end{equation*}
and hence $X'' \dashrightarrow X'''$ contains only flips. 
In this way we see that $X'_{1}$ and $X'''$ are isomorphic in codimension one. 

Since ${\rm lim}_{k\to \infty}a_{k}=0$, we have $t\geq a_{k}$ for any $k\gg 0$. 
Then 
$K_{X'''}+\Delta_{X'''}-a_{k}G_{X'''}$ 
is nef for any $k\gg0$ because $K_{X'''}+\Delta_{X'''}$ and $K_{X'''}+\Delta_{X'''}-tG_{X'''}$ are nef. 
Moreover $X'''$ and $X'_{k}$ are isomorphic in codimension one since $X'_{k}$ and $X'_{1}$ are isomorphic in codimension one and $X'_{1}$ and $X'''$ are isomorphic in codimension one. 
We recall that $(X'_{k},\Delta_{X'_{k}}-a_{k}G_{X'_{k}})$ is in particular a log minimal model of $(X,\Delta-a_{k}G)$, which is condition (ii) in Step \ref{step2klt}. 
From these facts, we see that $(X''',\Delta_{X'''}-a_{k}G_{X'''})$ is a log minimal model of $(X,\Delta-a_{k}G)$ for any $k\gg 0$. 
Let $p:W \to X$ and $q:W \to X'''$ be a common resolution of $X \dashrightarrow X'''$. 
Then for any $k\gg 0$ we have
$$p^{*}(K_{X}+\Delta-a_{k}G)-q^{*}(K_{X'''}+\Delta_{X'''}-a_{k}G_{X'''})\geq 0. $$
By considering the limit $k\to \infty$, we have 
$$p^{*}(K_{X}+\Delta)-q^{*}(K_{X'''}+\Delta_{X'''})\geq 0.$$ 
Since $K_{X'''}+\Delta_{X'''}$ is nef, we see that $(X''',\Delta_{X'''})$ is a weak lc model of $(X,\Delta)$. 
Therefore, by \cite[corollary 3.7]{birkar-flip}, $(X,\Delta)$ has a log minimal model. 
\end{step} 

\begin{step}\label{step4klt}
Finally we prove that Conjecture \ref{conjnonvanish} holds for $(X,\Delta)$. 
%By \cite[Theorem 4.1 (iii)]{birkar-flip} and running the $(K_{X}+\Delta)$-MMP with scaling of an ample divisor, we get a log minimal model $(X,\Delta)\dashrightarrow (X',\Delta_{X'})$. Then $\Delta_{X'}\neq 0$ and $\tau(X',0\,;\Delta_{X'})=1$. Therefore, by replacing $(X,\Delta)$ with $(X',\Delta_{X'})$, 
By running the $(K_{X}+\Delta)$-MMP with scaling of an ample divisor and replacing $(X,\Delta)$ with the resulting log minimal model, we can assume that $K_{X}+\Delta$ is nef. 
Pick a sufficiently small positive real number $t$ and run the $(K_{X}+(1-t)\Delta)$-MMP with scaling of an ample divisor. 
Since $K_{X}+(1-t)\Delta$ is not pseudo-effective we get a Mori fiber space 
$$(X,(1-t)\Delta)\dashrightarrow (X',(1-t)\Delta_{X'})\to Z.$$ 
Moreover, since $t$ is sufficiently small, $K_{X'}+\Delta_{X'}$ is trivial over $Z$ and Conjecture \ref{conjnonvanish} for $(X,\Delta)$ is equivalent to Conjecture \ref{conjnonvanish} for $(X',\Delta_{X'})$ (see \cite[Proposition 3.2]{birkar-existII}). 
Now we can easily check that Conjecture \ref{conjnonvanish} for $(X',\Delta_{X'})$ holds by Ambro's canonical bundle formula (cf.~\cite[Corollary 3.2]{fg-bundle}) and Conjecture \ref{conjnonvanish}$_{\leq d-1}$. 
So we are done. 
\end{step}
\end{proof}

Therefore we complete the proof of Theorem \ref{thmmain}. 
\end{proof}

\begin{rem}\label{rempethre}
Let $(X,\Delta)$ be a projective $\mathbb{Q}$-factorial log canonical pair such that $(X,0)$ is Kawamata log terminal, and let $A$ be a general sufficiently ample divisor. 
Suppose that $K_{X}+\Delta$ is pseudo-effective and $K_{X}+t\Delta$ is not pseudo-effective for any $t<1$. 
Then as in Step \ref{step1lc} in the proof of Case \ref{case1} (or Case \ref{case2}), we see that pseudo-effective threshold $\tau(X,t\Delta;A)$ is a linear function of $t$ when $t\in [\epsilon,1]$ for some $\epsilon>0$ sufficiently close to one. 
In the proof of Theorem \ref{thmmain} we construct lc-trivial fibrations by using this property. 

We introduce a simple application of the above property. 
Notations as above, assume $\Delta$ and $A$ are $\mathbb{Q}$-divisors. 
Then 
$$\mathcal{E}=\{(a,b)\in[0,1]\times[0,1]\mid K_{X}+a\Delta+bA {\rm \;is\;pseudo\mathchar`-effective}\}$$
is a rational polytope in $[0,1]\times[0,1]$. 
Indeed, if we set $\tau_{\epsilon}=\tau(X,\epsilon\Delta;A)$, where $\epsilon\in\mathbb{Q}$ is as above, then $\tau_{\epsilon}\in\mathbb{Q}$ (cf.~\cite[Proposition 8.7]{dhp}) and 
$$\mathcal{E}\cap\bigl([\epsilon ,1]\times[0,1]\bigr)=\{(a,b)\mid (1-\epsilon)b\geq \tau_{\epsilon}(1-a)\}$$
by the above property. 
On the other hand, if we pick a positive rational number $\epsilon'<\tau_{\epsilon}$, then $K_{X}+\epsilon\Delta+\epsilon'A$ is not pseudo-effective. 
Thus 
$$\mathcal{E}\cap\bigl([0, \epsilon]\times[0,1]\bigr)=\{(a,b+\epsilon')\mid K_{X}+(a\Delta+bA)+\epsilon'A {\rm \;is\;pseudo\mathchar`-effective}\}$$
is a rational polytope by \cite[Corollary 1.1.5]{bchm}. 
In this way we see that $\mathcal{E}$ is a rational polytope in $[0,1]\times[0,1]$. 
\end{rem}

\begin{proof}[Proof of Corollary \ref{cor1}]
It immediately follows from Theorem \ref{thmmain}. 
\end{proof}

\begin{proof}[Proof of Corollary \ref{cor2}]
It immediately follows from Theorem \ref{thmmain}. 
\end{proof}

%%%%%%%%%%%%%%%

\end{document}